\documentclass[12pt,reqno]{amsart}
\usepackage{etoolbox}
\makeatletter
\patchcmd\maketitle
  {\uppercasenonmath\shorttitle}
  {}
  {}{}
\patchcmd\maketitle
  {\@nx\MakeUppercase{\the\toks@}}
  {\the\toks@}
  {}
  {}{}
\patchcmd\@settitle{\uppercasenonmath\@title}{\Large}{}{}
\patchcmd\@setauthors
  {\MakeUppercase{\authors}}
  {\authors}
  {}{}
\makeatother
\usepackage{amsmath,amssymb,amsthm}
\usepackage{color}
\usepackage{url}
\usepackage{tikz-cd}
\usepackage[utf8]{inputenc}
\usepackage[T1]{fontenc}
\newtheorem{theorem}{Theorem}[section]

\newtheorem{corollary}{Corollary}[section]
\newtheorem{proposition}{Proposition}[section]
\newtheorem{lemma}{Lemma}[section]
\newtheorem{remark}{Remark}[section]

\numberwithin{equation}{section}
  \newtheorem{thqt}{Theorem}
  
  \newtheorem{lemqt}[thqt]{Lemma}

  \renewcommand{\thethqt}{\Alph{thqt}}
\usepackage[colorlinks=true]{hyperref}
\hypersetup{urlcolor=blue, citecolor=red , linkcolor= blue}
\usepackage[capitalise,noabbrev,nameinlink]{cleveref}      
\begin{document}
\address{$^{[1_a]}$ University of Monastir, Faculty of Economic Sciences and Management of Mahdia, Mahdia, Tunisia}
\address{$^{[1_b]}$ Laboratory Physics-Mathematics and Applications (LR/13/ES-22), Faculty of Sciences of Sfax, University of Sfax, Sfax, Tunisia}
\email{\url{kais.feki@fsegma.u-monastir.tn}\,;\,\url{kais.feki@hotmail.com}}



\date{\today}
\author[Kais Feki] {\Large{Kais Feki}$^{1_{a,b}}$}
\title[Improved inequalities related to the $A$-numerical radius for commutators of operators]{Improved inequalities related to the $A$-numerical radius for commutators of operators}

\maketitle

\begin{abstract}
Let $A$ be a positive bounded linear operator on a complex Hilbert space $\mathcal{H}$ and $\mathbb{B}_{A}(\mathcal{H})$ be the subspace of all operators which admit $A$-adjoints operators. In this paper, we establish some inequalities involving the commutator and the anticommutator of operators in semi-Hilbert spaces, i.e. spaces generated by positive semidefinite sesquilinear forms. Mainly, among other inequalities, we prove that for $T, S\in\mathbb{B}_{A}(\mathcal{H})$ we have
\begin{align*}
\omega_A(TS \pm ST) \leq 2\sqrt{2}\min\Big\{f_A(T,S), f_A(S,T) \Big\},
\end{align*}
where
$$f_A(X,Y)=\|Y\|_A\sqrt{\omega_A^2(X)-\frac{\left|\,\left\|\frac{X+X^{\sharp_A}}{2}\right\|_A^2-\left\|\frac{X-X^{\sharp_A}}{2i}\right\|_A^2\right|}{2}}.$$
This covers and improves the well-known inequalities of Fong and Holbrook. Here $\omega_A(\cdot)$ and $\|\cdot\|_A$ are the $A$-numerical radius and the $A$-operator seminorm of semi-Hilbert space operators, respectively and $X^{\sharp_A}$ denotes a distinguished $A$-adjoint operator of $X$.
\end{abstract}
\section{Introduction}
Let $\big(\mathcal{H}, \langle\cdot, \cdot\rangle\big)$ be a complex Hilbert space endowed with the norm $\|\cdot\|$. Let $\mathbb{B}(\mathcal{H})$ stand for the $C^*$-algebra of all bounded linear operators on $\mathcal{H}$ and $I$ denote the identity operator on $\mathcal{H}$. For every operator $T\in\mathbb{B}(\mathcal{H})$ its range is denoted by $\mathcal{R}(T)$, its null space by $\mathcal{N}(T)$, and its adjoint by $T^*$. An operator $T\in \mathbb{B}(\mathcal{H})$ is called positive if $\langle Tx, x\rangle\geq0$ for all $x\in{\mathcal H }$, and we then write $T\geq 0$. The square root of every positive operator $T$ is denoted by $T^{1/2}$. If $T\geq 0$, then the absolute value of $T$ is given by $|T|:=(T^*T)^{1/2}$. If $\mathcal{S}$ is a given linear subspace of $\mathcal{H}$, then $\overline{\mathcal{S}}$ stands for its closure in the norm topology of $\mathcal{H}$. Moreover, the orthogonal projection onto a closed linear subspace $\mathcal{S}$ of $\mathcal{H}$ is denoted by $P_{\mathcal{S}}$. Throughout this article, we suppose that $A\in\mathbb{B}(\mathcal{H})$ is a positive operator, which
induces the following semi-inner product
$$\langle\cdot,\cdot\rangle_{A}:\mathcal{H}\times \mathcal{H}\longrightarrow\mathbb{C},\;(x,y)\longmapsto \langle x, y\rangle_{A}:=\langle Ax, y\rangle=\langle A^{1/2}x, A^{1/2}y\rangle.$$
The seminorm induced by ${\langle \cdot, \cdot\rangle}_{A}$ is given by ${\|x\|}_A = \sqrt{{\langle x, x\rangle}_{A}}$ for every $x\in\mathcal{H}$. Clearly ${\|x\|}_{A} = 0$ if and only if $x\in \mathcal{N}(A)$ which implies that ${\|\cdot\|}_{A}$ is a norm on $\mathcal{H}$ if and only if $A$ is an injective operator. Further, it can be seen that the seminormed space $(\mathcal{H}, {\|\cdot\|}_{A})$ is complete if and only if $\mathcal{R}(A)$ is closed in $\mathcal{H}$. In this article we continue the line of research begun in \cite{feki03,fpreprint03,faiot}. Notice that the inspiration for
our investigation comes from the works of Kittaneh et al. \cite{A.K.1,A.K.2}.

The semi-inner product $\langle\cdot,\cdot\rangle_A$ induces an inner product on the quotient space $\mathcal{H}/\mathcal{N}(A)$ which is not complete unless $\mathcal{R}(A)$ is closed subspace in $\mathcal{H}$. However, a canonical construction due to L. de Branges and J. Rovnyak in \cite{branrov} shows that the completion of $\mathcal{H}/\mathcal{N}(A)$ is isometrically isomorphic to the Hilbert space $\mathcal{R}(A^{1/2})$
with the inner product
\begin{equation*}
\langle A^{1/2}x,A^{1/2}y\rangle_{\mathbf{R}(A^{1/2})}:=\langle P_{\overline{\mathcal{R}(A)}}x, P_{\overline{\mathcal{R}(A)}}y\rangle,\;\forall\, x,y \in \mathcal{H}.
\end{equation*}
 For the sequel, the Hilbert space $\big(\mathcal{R}(A^{1/2}), \langle\cdot,\cdot\rangle_{\mathbf{R}(A^{1/2})}\big)$ will be denoted by $\mathbf{R}(A^{1/2})$. For more details concerning the Hilbert space $\mathbf{R}(A^{1/2})$, we refer the reader to \cite{acg3} and the references therein.

For $T\in \mathbb{B}(\mathcal{H})$, an operator $S\in \mathbb{B}(\mathcal{H})$ is called an $A$-adjoint operator of $T$ if for every $x, y\in \mathcal{H}$,
we have ${\langle Tx, y\rangle}_{A} = {\langle x, Sy\rangle}_{A}$, that is, $S$ is solution of the operator equation $AX= T^*A$. This kind of equations can be studied by using the next theorem due to Douglas (for its proof see \cite{Dou}).
\begin{thqt}\label{doug}
If $T, S \in \mathbb{B}(\mathcal{H})$, then the following statements are equivalent:
\begin{itemize}
\item[{\rm (i)}] $\mathcal{R}(S) \subseteq \mathcal{R}(T)$.
\item[{\rm (ii)}] $TD=S$ for some $D\in \mathbb{B}(\mathcal{H})$.
\item[{\rm (iii)}] There exists  $\lambda> 0$ such that $\|S^*x\| \leq \lambda\|T^*x\|$ for all $x\in \mathcal{H}$.
\end{itemize}
If one of these conditions holds, then there exists a unique solution of the operator equation $TX=S$, denoted by $Q$, such that $\mathcal{R}(Q) \subseteq \overline{\mathcal{R}(T^{*})}$. Such $Q$ is called the reduced solution of $TX=S$.
\end{thqt}
If we denote by $\mathbb{B}_{A}(\mathcal{H})$ and $\mathbb{B}_{A^{1/2}}(\mathcal{H})$ the sets of all operators that admit $A$-adjoints and $A^{1/2}$-adjoints, respectively, then an application of Theorem \ref{doug} gives
\begin{align*}
\mathbb{B}_{A}(\mathcal{H}) = \big\{T\in\mathbb{B}(\mathcal{H})\,; \; \mathcal{R}(T^*A) \subseteq \mathcal{R}(A)\big\},
\end{align*}
and
\begin{align*}
\mathbb{B}_{A^{1/2}}(\mathcal{H}) = \big\{T\in \mathbb{B}(\mathcal{H})\,; \,\, \exists\, c>0\,;
\,\,{\|Tx\|}_{A} \le  c{\|x\|}_{A}, \,\, \forall x\in \mathcal{H}\big\}.
\end{align*}
Operators in $ \mathbb{B}_{A^{1/2}}(\mathcal{H})$ are called $A$-bounded. Notice that $\mathbb{B}_{A}(\mathcal{H})$ and $\mathbb{B}_{A^{1/2}}(\mathcal{H})$ are two subalgebras of $\mathbb{B}(\mathcal{H})$ which are, in general, neither closed nor dense in $\mathbb{B}(\mathcal{H})$ (see \cite{acg1}). Moreover, the following inclusions $\mathbb{B}_{A}(\mathcal{H})\subseteq \mathbb{B}_{A^{1/2}}(\mathcal{H})\subseteq \mathbb{B}(\mathcal{H})$ hold and are in general proper (see \cite{feki01}). The seminorm of an operator $T\in\mathbb{B}_{A^{1/2}}(\mathcal{H})$ is given by
\begin{align*}
{\|T\|}_{A}:=\sup_{\substack{x\in \overline{\mathcal{R}(A)},\\ x\not=0}}\frac{{\|Tx\|}_{A}}{{\|x\|}_{A}}
=\sup\big\{{\|Tx\|}_{A}\,; \,\, x\in\mathcal{H}, {\|x\|}_A =1\big\} < +\infty,
\end{align*}
(see \cite{feki01} and the references therein). Notice that it may happen that ${\|T\|}_A = + \infty$
for some $T\in\mathbb{B}(\mathcal{H})$ (see \cite[Example 2]{feki01}). It is not difficult to verify that, for $T\in\mathbb{B}_{A^{1/2}}(\mathcal{H})$,
we have ${\|Tx\|}_{A}\leq {\|T\|}_{A}{\|x\|}_{A}$ for all $x\in \mathcal{H}$. This yields that, for $T, S\in\mathbb{B}_{A^{1/2}}(\mathcal{H})$, we have ${\|TS\|}_{A}\leq {\|T\|}_{A}{\|S\|}_{A}$.

Recently, several extensions for the concept of the numerical radius have been investigated (for example see \cite{bc,pkbsm} and the reference therein). One of these extensions is the notion $A$-numerical radius of an operator $T\in \mathbb{B}(\mathcal{H})$ which was firstly introduced by Saddi in \cite{saddi} as
\begin{align*}
\omega_{A}(T) = \sup\Big\{\big|{\langle Tx, x\rangle}_A\big|\,; \,\,\, x\in \mathcal{H},\, {\|x\|}_A = 1\Big\}.
\end{align*}
It should be mention here that it may happen that $\omega_A(T) = + \infty$ for some $T\in\mathbb{B}(\mathcal{H})$ (see \cite{feki03}). However, $\omega_{A}(\cdot)$ defines a seminorm on $\mathbb{B}_{A^{1/2}}(\mathcal{H})$ which is equivalent to the $A$-operator seminorm ${\|\cdot\|}_{A}$. More precisely, for every $T\in \mathbb{B}_{A^{1/2}}(\mathcal{H})$, we have
\begin{align}\label{I.1.02}
\frac{1}{2}{\|T\|}_{A} \leq \omega_{A}(T)\leq {\|T\|}_{A}.
\end{align}
For an account of the results related to the $A$-numerical radius of $A$-bounded operators, the reader is referred to \cite{bakfeki01,bakfeki04,B.F.P, B.P.N, B.P.N.RIMA,BNP2,feki01, Feki.Sid}. The $A$-Crawford number of an operator $T$ is defined by
\begin{align*}
c_A(T) = \inf \big\{|{\langle Tx, x\rangle}_A|\,;\,\, x\in\mathcal{H},\,{\|x\|}_A =1\big\}.
\end{align*}
An operator $T\in\mathbb{B}(\mathcal{H})$ is called $A$-selfadjoint if $AT$ is selfadjoint. Further, an operator $T\in \mathbb{B}(\mathcal{H})$ is said $A$-positive if $AT$ is positive and we note $T\geq_A 0$. Trivially, an $A$-positive operator is always an $A$-selfadjoint operator since $\mathcal{H}$ is a complex Hilbert space. Moreover, it was shown in \cite{feki01} that if $T$ is $A$-self-adjoint, then
\begin{equation}\label{aself1}
\|T\|_{A}=\omega_A(T).
\end{equation}

If $T\in\mathbb{B}_{A}(\mathcal{H})$, then the reduced solution to the operator equation $AX = T^*A$ is
denoted by $T^{\sharp_A}$. Notice that $T^{\sharp_A} = A^{\dag}T^*A$, where $A^{\dag}$ is the Moore--Penrose inverse of $A$, which
is the unique linear mapping from $\mathcal{R}(A) \oplus \mathcal{R}(A)^{\perp}$ into $\mathcal{H}$ satisfying the following Moore--Penrose equations:
\begin{align*}
AXA = A, \,\, XAX = X, \,\, XA = P_{\overline{\mathcal{R}(A)}}
\,\, \text{ and } \,\, AX = P_{\overline{\mathcal{R}(A)}}|_{\mathcal{R}(A) \oplus {\mathcal{R}(A)}^{\perp}}.
\end{align*}
We mention here that if $T\in\mathbb{B}_{A}(\mathcal{H})$, then $T^{\sharp_A}\in\mathbb{B}_{A}(\mathcal{H})$,
$(T^{\sharp_A})^{\sharp_A} = P_{\overline{\mathcal{R}(A)}}TP_{\overline{\mathcal{R}(A)}}$ and
$\big((T^{\sharp_A})^{\sharp_A}\big)^{\sharp_A} = T^{\sharp_A}$. Moreover, if $T, S\in\mathbb{B}_{A}(\mathcal{H})$, then $TS\in\mathbb{B}_{A}(\mathcal{H})$ and $(TS)^{\sharp_A} = S^{\sharp_A}T^{\sharp_A}$. Notice that for $T\in\mathbb{B}_{A}(\mathcal{H})$, we have $T^{\sharp_A}T\geq_A 0$ and  $TT^{\sharp_A}\geq_A 0$. Moreover, by using \eqref{aself1}, we see that
\begin{align}\label{I.1.01}
{\|T^{\sharp_A}T\|}_A = {\|TT^{\sharp_A}\|}_A = {\|T\|}^2_A = {\|T^{\sharp_A}\|}^2_A.
\end{align}
For other results covering $T^{\sharp_A}$, we invite the reader to consult \cite{acg2, acg1, Ma.Se.Su}.

Recently, several improvements of the inequalities in \eqref{I.1.02} have been given (see \cite{B.P.N,feki03}). For instance,
it has been shown in \cite{fpreprint03}, that if $T\in\mathbb{B}_{A}(\mathcal{H})$, then
  \begin{equation}\label{feki2020}
  \tfrac{1}{4}\|T^{\sharp_A} T+TT^{\sharp_A}\|_A\le  \omega_A^2(T) \le \tfrac{1}{2}\|T^{\sharp_A} T+TT^{\sharp_A}\|_A.
  \end{equation}
Recently, the present proved in \cite{fpreprint03} that for every $T, S\in\mathbb{B}_{A}(\mathcal{H})$, we have
\begin{align}\label{but}
\omega_A(TS \pm ST)
& \leq 2\sqrt{2}\min\Big\{{\|T\|}_A \omega_A(S), {\|S\|}_A \omega_A(T)\Big\}.
\end{align}
Of course, if $A=I$, we get the well-known inequalities of Fong and Holbrook (see \cite{fonghol}). One main target of this paper is to generalize \eqref{but}. Also, a considerable refinement of \eqref{but} is established.

\section{Results}
In this section, we prove our results. In order to prove our first result in this section, we need the following lemmas.

\begin{lemma}(\cite{faiot})\label{lawil}
Let $T_1, T_2,S_1,S_2\in\mathbb{B}_{A}(\mathcal{H})$. Then
\begin{align*}
\omega_A(T_1S_1 \pm S_2T_2) &\leq \sqrt{{\big\|T_1T_1^{\sharp_A} + T_2^{\sharp_A} T_2\big\|}_A}\sqrt{{\big\|S_1^{\sharp_A} S_1+S_2S_2^{\sharp_A}\big\|}_A}.
\end{align*}
\end{lemma}

\begin{lemma}(\cite{faiot})\label{lemma1}
Let $T, S\in\mathbb{B}_{A^{1/2}}(\mathcal{H})$ and $\mathbb{A}=\begin{pmatrix}
A & 0 \\
0 & A
\end{pmatrix}$. Then, the following assertions hold
\begin{itemize}
\item[(i)] $\omega_{\mathbb{A}}\left[\begin{pmatrix}
T & 0 \\
0 & S
\end{pmatrix}\right] = \max\big\{\omega_{A}(T), \omega_{A}(S)\big\}$.
\item[(ii)] $\omega_{\mathbb{A}}\left[\begin{pmatrix}
0 & T \\
T & 0
\end{pmatrix}\right] = \omega_{A}(T)$.
\item[(iii)] ${\left\|\begin{pmatrix}
T & 0 \\
0 & S
\end{pmatrix}\right\|}_{\mathbb{A}}={\left\|\begin{pmatrix}
0 & T \\
S & 0
\end{pmatrix}\right\|}_{\mathbb{A}} = \max\big\{{\|T\|}_{A}, {\|S\|}_{A}\big\}$.
\end{itemize}
\end{lemma}
Our first result in this paper reads as follows.
\begin{theorem}\label{T.2.15}
Let $T_1, T_2,S\in\mathbb{B}_{A}(\mathcal{H})$. Then
\begin{align*}
\omega_A(T_1S \pm ST_2)
& \leq 4\,\omega_{\mathbb{A}}\left[\begin{pmatrix}
0&T_1\\
T_2 &0
\end{pmatrix}\right]\omega_A(S).
\end{align*}
\end{theorem}
\begin{proof}
By letting $S_1=S_2=S$ in Lemma \ref{lawil} we obtain
\begin{align*}
\omega_A(T_1S \pm ST_2) &\leq \sqrt{{\big\|T_1T_1^{\sharp_A} + T_2^{\sharp_A} T_2\big\|}_A}\sqrt{{\big\|S^{\sharp_A} S+SS^{\sharp_A}\big\|}_A}.
\end{align*}
Moreover, by using the first inequality in \eqref{feki2020} we get
\begin{align}\label{lab}
\omega_A(T_1S \pm ST_2) \leq 2\sqrt{{\big\|T_1T_1^{\sharp_A} + T_2^{\sharp_A} T_2\big\|}_A}\,\omega_A(S).
\end{align}
Let $\mathbb{T}=\begin{pmatrix}
0&T_1\\
T_2 &0
\end{pmatrix}$. It can be observed that
$$\mathbb{T}^{\sharp_\mathbb{A}}\mathbb{T}+\mathbb{T}\mathbb{T}^{\sharp_\mathbb{A}}=\begin{pmatrix}
T_1T_1^{\sharp_A} + T_2^{\sharp_A} T_2&0\\
0 &T_1^{\sharp_A}T_1 + T_2 T_2^{\sharp_A}
\end{pmatrix}$$
This implies, by Lemma \ref{lemma1}(iii), that
$$\left\|\mathbb{T}^{\sharp_\mathbb{A}}\mathbb{T}+\mathbb{T}\mathbb{T}^{\sharp_\mathbb{A}}\right\|_\mathbb{A}=\max\left\{{\big\|T_1T_1^{\sharp_A} + T_2^{\sharp_A} T_2\big\|}_A, {\big\|T_1^{\sharp_A}T_1 + T_2 T_2^{\sharp_A}\big\|}_A \right\}.$$
So, by taking into consideration \eqref{lab} we see that
\begin{align*}
\omega_A(T_1S \pm ST_2)
&\leq 2\sqrt{{\big\|T_1T_1^{\sharp_A} + T_2^{\sharp_A} T_2\big\|}_A}\,\omega_A(S)\\
&\leq 2\sqrt{\left\|\mathbb{T}^{\sharp_\mathbb{A}}\mathbb{T}+\mathbb{T}\mathbb{T}^{\sharp_\mathbb{A}}\right\|_\mathbb{A}}\,\omega_A(S)\\
&\leq 4\,\omega_{\mathbb{A}}(\mathbb{T})\,\omega_A(S),
\end{align*}
where the last inequality follows from the first inequality in \eqref{feki2020}. Hence the proof is complete.
\end{proof}

The following corollary is an immediate consequence of Theorem \ref{T.2.15} and  provides an upper bound for the $A$-numerical radius of the commutator $TS-ST$ when $T$ and $S$ are in $\mathbb{B}_{A}(\mathcal{H})$.
\begin{corollary}
Let $T, S\in\mathbb{B}_{A}(\mathcal{H})$. Then
\begin{align}\label{sss1}
\omega_A(TS \pm ST) &\leq 4\omega_A(T)\,\omega_A(S).
\end{align}
Moreover, if $TS=ST$, then
\begin{align}\label{sss2}
\omega_A(TS) &\leq 2\omega_A(T)\,\omega_A(S).
\end{align}
\end{corollary}
\begin{proof}
By letting $T_1=T_2=T$ in Theorem \ref{T.2.15} and then using Lemma \ref{lemma1}(i) we reach \eqref{sss1}. Further, \eqref{sss2} follows immediately from \eqref{sss1}.
\end{proof}

Now, we aim to prove a generalization of \eqref{but}. In order to achieve this goal, we need to prove the following result.
\begin{proposition}\label{immmmm}
Let $T,S\in \mathbb{B}_A(\mathcal{H})$. Then,
\begin{equation}\label{omprovenew}
\| TT^{\sharp_A}+S^{\sharp_A}S\|_A\leq \max\left\{\|T\|_A^2,\|S\|_A^2\right\}+\|ST\|_A.
\end{equation}
\end{proposition}
To prove \eqref{omprovenew}, we shall recall the following lemma.
\begin{lemqt}(\cite{acg3,Ma.Se.Su})\label{alem2}
	Let $T\in \mathbb{B}(\mathcal{H})$. Then $T\in \mathbb{B}_{A^{1/2}}(\mathcal{H})$ if and only if there exists a unique $\widetilde{T}\in \mathbb{B}(\mathbf{R}(A^{1/2}))$ such that $Z_AT =\widetilde{T}Z_A$. Here, $Z_{A}: \mathcal{H} \rightarrow \mathbf{R}(A^{1/2})$ is
	defined by $Z_{A}x = Ax$. Moreover, the following properties hold
\begin{itemize}
  \item [(i)] $\|T\|_A=\|\widetilde{T}\|_{\mathbb{B}(\mathbf{R}(A^{1/2}))}$.
  \item [(ii)]  $\widetilde{T^{\sharp_A}}=(\widetilde{T})^*.$
\end{itemize}
\end{lemqt}

Now, we are ready to prove Proposition \ref{immmmm}.

\begin{proof}[Proof of Proposition \ref{immmmm}]
Since, $T,S\in \mathbb{B}_{A^{1/2}}(\mathcal{H})$, then by Lemma \ref{alem2} there exists two unique operators $\widetilde{T},\widetilde{S}\in \mathbb{B}(\mathbf{R}(A^{1/2}))$ such that $Z_AT=\widetilde{T}Z_A$ and $Z_AS=\widetilde{S}Z_A$. Moreover, clearly $T+S$ and $TS$ are in $\mathbb{B}_{A^{1/2}}(\mathcal{H})$. An application of \cite[Lemma 2.1.]{fekilaa} gives
\begin{equation}\label{prosum}
\widetilde{TS}=\widetilde{T}\widetilde{S}\;\text{ and }\; \widetilde{T+S}=\widetilde{T}+\widetilde{S}.
\end{equation}
Now, by using Lemma \ref{alem2} together with \eqref{prosum}, it can be seen that
\begin{align}\label{firr}
\|S^{\sharp_A} S+TT^{\sharp_A}\|_A
& =\|\widetilde{S^{\sharp_A} S+TT^{\sharp_A}}\|_{\mathbb{B}(\mathbf{R}(A^{1/2}))} \nonumber\\
 &=\|(\widetilde{S})^{*}\widetilde{S}+\widetilde{T}(\widetilde{T})^{*}\|_{\mathbb{B}(\mathbf{R}(A^{1/2}))}.
\end{align}
Moreover, by using basic properties of the spectral radius of Hilbert space operators, we see that
\begin{align*}
\|(\widetilde{S})^{*}\widetilde{S}+\widetilde{T}(\widetilde{T})^{*}\|_{\mathbb{B}(\mathbf{R}(A^{1/2}))}
& =r\left((\widetilde{S})^{*}\widetilde{S}+\widetilde{T}(\widetilde{T})^{*}\right)\\
 &=r\left[\begin{pmatrix}
(\widetilde{S})^{*}\widetilde{S}+\widetilde{T}(\widetilde{T})^{*}&0 \\
0&0
\end{pmatrix}\right]\\
&=r\left[\begin{pmatrix}
|\widetilde{S}|&|(\widetilde{T})^{*}| \\
0&0
\end{pmatrix}
\begin{pmatrix}
|\widetilde{S}|&0 \\
|(\widetilde{T})^{*}|&0
\end{pmatrix}
\right]\\
&=r\left[
\begin{pmatrix}
|\widetilde{S}|&0 \\
|(\widetilde{T})^{*}|&0
\end{pmatrix}
\begin{pmatrix}
|\widetilde{S}|&|(\widetilde{T})^{*}|\\
0&0
\end{pmatrix}
\right]
\end{align*}
Hence, we get
\begin{align*}
\|(\widetilde{S})^{*}\widetilde{S}+\widetilde{T}(\widetilde{T})^{*}\|_{\mathbb{B}(\mathbf{R}(A^{1/2}))}=r\left[
\begin{pmatrix}
(\widetilde{S})^{*}\widetilde{S}& |\widetilde{S}|\,|(\widetilde{T})^{*}|\\
|(\widetilde{T})^{*}|\,|\widetilde{S}|&\widetilde{T}(\widetilde{T})^{*}
\end{pmatrix}
\right]
\end{align*}
Thus, by using \cite[Theorem 1.1.]{hou} we obtain
\begin{align*}
&\|(\widetilde{S})^{*}\widetilde{S}+\widetilde{T}(\widetilde{T})^{*}\|_{\mathbb{B}(\mathbf{R}(A^{1/2}))}\\
&\leq r\left[
\begin{pmatrix}
\|(\widetilde{S})^{*}\widetilde{S}\|_{\mathbb{B}(\mathbf{R}(A^{1/2}))}& \|\,|\widetilde{S}|\,|(\widetilde{T})^{*}|\,\|_{\mathbb{B}(\mathbf{R}(A^{1/2}))}\\
\|\,|(\widetilde{T})^{*}|\,|\widetilde{S}|\,\|_{\mathbb{B}(\mathbf{R}(A^{1/2}))}&\|\widetilde{T}(\widetilde{T})^{*}\|_{\mathbb{B}(\mathbf{R}(A^{1/2}))}
\end{pmatrix}
\right]\\
&=\left\|
\begin{pmatrix}
\|\widetilde{S}\|_{\mathbb{B}(\mathbf{R}(A^{1/2}))}^2& \|\widetilde{S}\widetilde{T}\|_{\mathbb{B}(\mathbf{R}(A^{1/2}))}\\
\|\widetilde{S}\widetilde{T}\|_{\mathbb{B}(\mathbf{R}(A^{1/2}))}&\|\widetilde{T}\|_{\mathbb{B}(\mathbf{R}(A^{1/2}))}^2
\end{pmatrix}
\right\|,
\end{align*}
where the last equality follows since $$\|\,|\widetilde{S}|\,|(\widetilde{T})^{*}|\,\|_{\mathbb{B}(\mathbf{R}(A^{1/2}))}=\|\,|(\widetilde{T})^{*}|\,|\widetilde{S}|\,\|_{\mathbb{B}(\mathbf{R}(A^{1/2}))}=
\|\widetilde{S}\widetilde{T}\|_{\mathbb{B}(\mathbf{R}(A^{1/2}))}.$$
So, we infer that
\begin{align*}
&\|(\widetilde{S})^{*}\widetilde{S}+\widetilde{T}(\widetilde{T})^{*}\|_{\mathbb{B}(\mathbf{R}(A^{1/2}))}\\
&\leq \left\|
\begin{pmatrix}
\|\widetilde{S}\|_{\mathbb{B}(\mathbf{R}(A^{1/2}))}^2& 0\\
0 &\|\widetilde{T}\|_{\mathbb{B}(\mathbf{R}(A^{1/2}))}^2
\end{pmatrix}
\right\|+\left\|
\begin{pmatrix}
0& \|\widetilde{S}\widetilde{T}\|_{\mathbb{B}(\mathbf{R}(A^{1/2}))}\\
\|\widetilde{S}\widetilde{T}\|_{\mathbb{B}(\mathbf{R}(A^{1/2}))}&0
\end{pmatrix}
\right\|\\
&=\max\left\{\|\widetilde{S}\|_{\mathbb{B}(\mathbf{R}(A^{1/2}))}^2,\|\widetilde{T}\|_{\mathbb{B}(\mathbf{R}(A^{1/2}))}^2\right\}+\|\widetilde{S}\widetilde{T}\|_{\mathbb{B}(\mathbf{R}(A^{1/2}))},
\end{align*}
where the last equality follows from Lemma \ref{lemma1}(iii) by letting $A=I$. So, by taking into account \eqref{firr} and then applying Lemma \ref{alem2}(i) we get
\begin{equation*}
\| TT^{\sharp_A}+S^{\sharp_A}S\|_A\leq \max\left\{\|T\|_A^2,\|S\|_A^2\right\}+\|ST\|_A.
\end{equation*}
This achieves the proof.
\end{proof}

Now, we are in a position to prove the following theorem which generalizes \eqref{but}.
\begin{theorem}\label{a5ir}
Let $T, S,X,Y\in\mathbb{B}_{A}(\mathcal{H})$ and $\mathbb{A}=\begin{pmatrix}
A &0\\
0 &A
\end{pmatrix}$. Then
\begin{align}\label{ccdd}
\omega_A(TX \pm YS)
&\leq 2\sqrt{\max\left\{\|T\|_A^2,\|S\|_A^2\right\}+\|ST\|_A}\sqrt{\omega_\mathbb{A}^2\left[ \begin{pmatrix}
0&X\\
Y &0
\end{pmatrix}\right]-\frac{1}{2}c_A(YX) }\nonumber\\
& \leq 2\sqrt{2}\max\left\{\|T\|_A,\|S\|_A\right\}\omega_{\mathbb{A}}\left[\begin{pmatrix}
0&X\\
Y &0
\end{pmatrix}\right].
\end{align}
\end{theorem}
\begin{proof}
Notice first that, it was shown in \cite{fpreprint03} that
\begin{equation}\label{T.2.16}
\omega_{\mathbb{A}}\left[\begin{pmatrix}
0 &X\\
Y &0
\end{pmatrix}\right]\geq \frac{1}{2}\sqrt{{\big\|X^{\sharp_A}X+YY^{\sharp_A}\big\|}_A + 2c_A(YX)}.
\end{equation}
Now, by applying Lemma \ref{lawil} together with \eqref{omprovenew} we obtain
\begin{align*}
\omega_A(TX \pm YS)
&\leq \sqrt{{\big\|TT^{\sharp_A} + S^{\sharp_A} S\big\|}_A}\sqrt{{\big\|X^{\sharp_A} X+YY^{\sharp_A}\big\|}_A}\\
&\leq \sqrt{\max\left\{\|T\|_A^2,\|S\|_A^2\right\}+\|ST\|_A}\sqrt{{\big\|X^{\sharp_A} X+YY^{\sharp_A}\big\|}_A}\\
&\leq 2\sqrt{\max\left\{\|T\|_A^2,\|S\|_A^2\right\}+\|ST\|_A}\sqrt{\omega_\mathbb{A}^2\left[ \begin{pmatrix}
0&X\\
Y &0
\end{pmatrix}\right]-\frac{1}{2}c_A(YX) },
\end{align*}
where the last inequality follows from \eqref{T.2.16}. On the other hand, we see that
\begin{align*}
\|ST\|_A
&\leq \|T\|_A\|S\|_A \leq\frac{1}{2}\left(\|T\|_A^2+\|S\|_A^2\right)\leq \max\left\{\|T\|_A^2,\|S\|_A^2\right\}.
\end{align*}
This immediately proves \eqref{ccdd}.
\end{proof}
\begin{remark}
By replacing $S$ by $T$ and $X,Y$ by $S$ in \eqref{ccdd} and then using Lemma \ref{lemma1}(ii) we get
\begin{align*}
\omega_A(TS \pm ST)
& \leq 2\sqrt{2}{\|T\|}_A \omega_A(S).
\end{align*}
So, by changing the roles between $T$ and $S$ in the last inequality we reach \eqref{but}.
\end{remark}
For the rest of this paper, for any arbitrary operator $T\in {\mathcal B}_A({\mathcal H})$, we write
$$\Re_A(T):=\frac{T+T^{\sharp_A}}{2}\;\;\text{ and }\;\;\Im_A(T):=\frac{T-T^{\sharp_A}}{2i}.$$
Our next aim is to improve the inequality \eqref{but}. To do this we need the following lemma.
\begin{lemma}\label{m2}
Let $T\in \mathbb{B}_{A}(\mathcal{H})$ be such that $\omega_A(T)\leq 1$. Then, for every $x\in \mathcal{H}$ with $\|x\|_A=1$ we have
\begin{equation}
\|Tx\|_A^2+\|T^{\sharp_A}x\|_A^2\leq 4\left(1-\frac{\left|\,\|\Re_A(T)\|_A^2-\|\Im_A(T)\|_A^2 \right|}{2} \right).
\end{equation}
\end{lemma}
In order to prove Lemma \ref{m2}, we first prove the following result.
\begin{lemqt}\label{ffii}
Let $T,S\in \mathbb{B}_{A}(\mathcal{H})$. Then,
\begin{align*}
\|T^{\sharp_A}T+TT^{\sharp_A}\|_A\leq 4\max\left\{\|\Re_A(T)\|_A^2 , \|\Im_A(T)\|_A^2\right\}-2\left|\,\|\Re_A(T)\|_A^2-\|\Im_A(T)\|_A^2 \right|.
\end{align*}
\end{lemqt}
\begin{proof}
Notice first that it was shown in \cite[Lemma 2.18]{rout} that
\begin{align}\label{j}
\|X^{\sharp_A}X+Y^{\sharp_A}Y\|_A
&\leq \max\left\{\|X+Y\|_A^2 , \|X-Y\|_A^2\right\}-\frac{\left|\,\|X+Y\|_A^2-\|X-Y\|_A^2 \right|}{2},
\end{align}
for every $X,Y\in \mathbb{B}_{A}(\mathcal{H})$. By replacing $X$ and $Y$ by $(T^{\sharp_A})^{\sharp_A}$ and $T^{\sharp_A}$ in \eqref{j} respectively we get
\begin{align*}
&\|(T^{\sharp_A}T+TT^{\sharp_A})^{\sharp_A}\|_A\\
&\leq \max\left\{\|(T^{\sharp_A}+T)^{\sharp_A}\|_A^2 , \|(T^{\sharp_A}-T)^{\sharp_A}\|_A^2\right\}-\frac{\left|\,\|(T^{\sharp_A}+T)^{\sharp_A}\|_A^2- \|(T^{\sharp_A}-T)^{\sharp_A}\|_A^2 \right|}{2}\\
&= \max\left\{\|T^{\sharp_A}+T\|_A^2 , \|T^{\sharp_A}-T\|_A^2\right\}-\frac{\left|\,\|T^{\sharp_A}+T\|_A^2- \|T^{\sharp_A}-T\|_A^2 \right|}{2}\\
&= 4\max\left\{\|\Re_A(T)\|_A^2 , \|\Im_A(T)\|_A^2\right\}-2\left|\,\|\Re_A(T)\|_A^2-\|\Im_A(T)\|_A^2 \right|.
\end{align*}
This shows the desired result since $\|R\|_A=\|R^{\sharp_A}\|_A$ for all $R\in \mathbb{B}_{A}(\mathcal{H})$.
\end{proof}
Now, we are ready to prove Lemma \ref{m2}.
\begin{proof}[Proof of Lemma \ref{m2}]
Let $x\in \mathcal{H}$ be such that $\|x\|_A=1$. By using the Cauchy-Schwarz inequality we see that
\begin{align*}
\|Tx\|_A^2+\|T^{\sharp_A}x\|_A^2
& =\langle(T^{\sharp_A} T+TT^{\sharp_A})x , x\rangle_A\\
 &\leq \omega_A\left(T^{\sharp_A} T+TT^{\sharp_A}\right)\\
& = \|T^{\sharp_A} T+TT^{\sharp_A}\|_A,
\end{align*}
where the last equality follows from \eqref{aself1} since $T^{\sharp_A} T+TT^{\sharp_A}\geq_A 0$. On the other hand, since $\Re_A(T)$ and $\Im_A(T)$ are $A$-selfadjoint operators, then by \eqref{aself1} we have
$$
\omega_A\Big(\Re_A(T)\Big)=\|\Re_A(T)\|_A\;\text{ and }\;\omega_A\Big(\Im_A(T)\Big)=\|\Im_A(T)\|_A.
$$
So, by applying Lemma \ref{ffii} it can be observed that
\begin{align*}
&\|T^{\sharp_A}T+TT^{\sharp_A}\|_A\\
&\leq 4\max\left\{\|\Re_A(T)\|_A^2 , \|\Im_A(T)\|_A^2\right\}-2\left|\,\|\Re_A(T)\|_A^2-\|\Im_A(T)\|_A^2 \right|\\
&= 4\max\left\{\omega_A^2\Big(\Re_A(T)\Big),\omega_A^2\Big(\Im_A(T)\Big)\right\}-2\left|\,\|\Re_A(T)\|_A^2-\|\Im_A(T)\|_A^2 \right|.
\end{align*}
On the other hand, it is not difficult to verify that
$$
\omega_A\Big(\Re_A(T)\Big)\leq \omega_A(T)\;\text{ and }\;\omega_A\Big(\Im_A(T)\Big)\leq \omega_A(T).
$$
This implies that
\begin{align*}
\|T^{\sharp_A}T+TT^{\sharp_A}\|_A
&\leq 4\omega_A^2(T)-2\left|\,\|\Re_A(T)\|_A^2-\|\Im_A(T)\|_A^2 \right|\\
&\leq 4\left(1-\frac{\left|\,\|\Re_A(T)\|_A^2-\|\Im_A(T)\|_A^2 \right|}{2} \right),
\end{align*}
where the last inequality follows since $\omega_A(T)\leq1$.
\end{proof}
Now, we are in a position to prove the following theorem.
\begin{theorem}\label{rm2}
Let $T,S,X,Y\in \mathbb{B}_{A}(\mathcal{H})$. Then
\begin{align}
&\omega_A(TXS\pm SYT)\nonumber\\
&\leq 2\sqrt{2}\|S\|_A\max\left\{\|X\|_A , \|Y\|_A\right\}\sqrt{\omega_A^2(T)-\frac{\left|\,\|\Re_A(T)\|_A^2-\|\Im_A(T)\|_A^2 \right|}{2}}\;.
\end{align}
\end{theorem}
\begin{proof}
Assume first that $\omega_A(T)\leq 1$, $\|X\|_A\leq 1$ and $\|Y\|_A\leq 1$. Let $x\in \mathcal{H}$ be such that $\|x\|_A=1$. By applying the Cauchy-Schwarz inequality we see that
\begin{align*}
|\langle (TX\pm YT)x, x\rangle_A|
&\leq |\langle Xx, T^{\sharp_A}x\rangle_A|+|\langle Tx, Y^{\sharp_A}x\rangle_A| \\
&\leq \|Xx\|_A\|T^{\sharp_A}x\|_A+\|Tx\|_A\|Y^{\sharp_A}x\|_A\\
&\leq \|X\|_A\|T^{\sharp_A}x\|_A+\|Tx\|_A\|Y^{\sharp_A}\|_A\\
&\leq \|T^{\sharp_A}x\|_A+\|Tx\|_A\\
&\leq \sqrt{2}\left(\|T^{\sharp_A}x\|_A^2+\|Tx\|_A^2\right)^{\frac{1}{2}}.
\end{align*}
So, by Lemma \ref{m2} we get
\begin{equation*}
|\langle (TX\pm YT)x, x\rangle_A|\leq 2\sqrt{2}\sqrt{1-\frac{\left|\,\|\Re_A(T)\|_A^2-\|\Im_A(T)\|_A^2 \right|}{2}}.
\end{equation*}
Thus, by taking the supremum over all $x\in \mathcal{H}$ with $\|x\|_A=1$ in the above inequality we get
\begin{equation}\label{3.6}
\omega_A(TX\pm YT)\leq 2\sqrt{2}\sqrt{1-\frac{\left|\,\|\Re_A(T)\|_A^2-\|\Im_A(T)\|_A^2 \right|}{2}}.
\end{equation}

Now, let $T,X,Y\in \mathbb{B}_{A}(\mathcal{H})$ be any operators. If $\max\left\{\|X\|_A , \|Y\|_A\right\}=0$ or $\omega_A(T)=0$, then obviously the desired result holds. Assume that $\omega_A(T)\neq0$ and $\max\left\{\|X\|_A,\|Y\|_A\right\}\neq0$. By replacing $T$, $X$ and $Y$ by $\frac{T}{\omega_A(T)}$, $\frac{X}{\max\left\{\|X\|_A,\|Y\|_A\right\}}$ and $\frac{Y}{\max\left\{\|X\|_A,\|Y\|_A\right\}}$ respectively in \eqref{3.6} we see that
\begin{align}\label{23.33}
&\omega_A(TX\pm YT)\nonumber\\
&\leq 2\sqrt{2}\max\left\{\|X\|_A , \|Y\|_A\right\}\omega_A(T)\sqrt{1-\frac{\left|\,\left\|\Re_A\left(\frac{T}{\omega_A(T)}\right)\right\|_A^2-\left\|\Im_A\left(\frac{T}{\omega_A(T)}\right)\right\|_A^2 \right|}{2}}\nonumber\\
 &=2\sqrt{2}\max\left\{\|X\|_A , \|Y\|_A\right\}\sqrt{\omega_A^2(T)-\frac{\left|\,\left\|\Re_A\left(T\right)\right\|_A^2-\left\|\Im_A\left(T\right)\right\|_A^2 \right|}{2}}.
\end{align}
By replacing $X$ and $Y$ by $XS$ and $SY$ respectively in the inequality \eqref{23.33}, we obtain
\begin{align*}
&\omega_A(TXS\pm SYT)\\
 &\leq 2\sqrt{2}\max\left\{\|XS\|_A , \|SY\|_A\right\}\sqrt{\omega_A^2(T)-\frac{\left|\,\|\Re_A\left(T\right)\|_A^2-\|\Im_A\left(T\right)\|_A^2 \right|}{2}}\\
  &\leq 2\sqrt{2}\|S\|_A\max\left\{\|X\|_A , \|Y\|_A\right\}\sqrt{\omega_A^2(T)-\frac{\left|\,\|\Re_A\left(T\right)\|_A^2-\|\Im_A\left(T\right)\|_A^2 \right|}{2}}.
\end{align*}
This proves the required result.
\end{proof}

The following result is an immediate consequence of Theorem \ref{rm2} and extends a recent result of Hirzallah and Kittaneh (see \cite{A.K.2}). Moreover, the obtained inequality considerably refine the inequality \eqref{but}.
\begin{theorem}\label{impr2020}
Let $T, S\in\mathbb{B}_{A}(\mathcal{H})$. Then
\begin{align*}
\omega_A(TS \pm ST) &\leq 2\sqrt{2}\min\Big\{f_A(T,S), f_A(S,T) \Big\}.
\end{align*}
where
$$f_A(X,Y)=\|Y\|_A\sqrt{\omega_A^2(X)-\frac{\left|\,\left\|\Re_A\left(X\right)\right\|_A^2-\left\|\Im_A\left(X\right)\right\|_A^2\right|}{2}}.$$
\end{theorem}
\begin{proof}
By letting $X=Y=I$ in Theorem \ref{rm2} we get
\begin{align}\label{q}
\omega_A(TS\pm ST)
&\leq 2\sqrt{2}\|S\|_A\sqrt{\omega_A^2(T)-\frac{\left|\,\|\Re_A(T)\|_A^2-\|\Im_A(T)\|_A^2 \right|}{2}}\;.
\end{align}
Now, by replacing $T$ and $S$ by $S$ and $T$ respectively in \eqref{q} we get the desired result.
\end{proof}

\begin{corollary}
Let $T\in\mathbb{B}_{A}(\mathcal{H})$. Then
\begin{align*}
\omega_A(T^2)\leq \sqrt{2}\|T\|_A\sqrt{\omega_A^2(T)-\frac{\left|\,\left\|\Re_A\left(T\right)\right\|_A^2-\left\|\Im_A\left(T\right)\right\|_A^2\right|}{2}}\;.
\end{align*}
\end{corollary}
\begin{proof}
Follows immediately by letting $T=S$ in Theorem \ref{impr2020}.
\end{proof}
\begin{corollary}
Let $T, S\in\mathbb{B}_{A}(\mathcal{H})$ be such that $\omega_A(TS \pm ST)=2\sqrt{2}\|S\|_A\omega_A(T)$ and $AS\neq 0$. Then
\begin{equation}\label{fin}
\left\|\Re_A\left(T\right)\right\|_A=\left\|\Im_A\left(T\right)\right\|_A.
\end{equation}
\end{corollary}
\begin{proof}
It follows from Theorem \ref{impr2020} that
\begin{align*}
\omega_A(TS\pm ST)
&\leq 2\sqrt{2}\|S\|_A\sqrt{\omega_A^2(T)-\frac{\left|\,\|\Re_A(T)\|_A^2-\|\Im_A(T)\|_A^2 \right|}{2}}\;\\
&\leq 2\sqrt{2}\|S\|_A\omega_A(T).
\end{align*}
So, since $\omega_A(TS \pm ST)=2\sqrt{2}\|S\|_A\omega_A(T)0$, then
\begin{align*}
2\sqrt{2}\|S\|_A\sqrt{\omega_A^2(T)-\frac{\left|\,\|\Re_A(T)\|_A^2-\|\Im_A(T)\|_A^2 \right|}{2}}=2\sqrt{2}\|S\|_A\omega_A(T).
\end{align*}
Since $AS\neq 0$, then $\|S\|_A\neq 0$. This immediately proves \eqref{fin} as desired.
\end{proof}


\begin{thebibliography}{99} 
\footnotesize
\bibitem{A.K.1} A. Abu-Omar and F. Kittaneh,
\textit{Numerical radius inequalities for products and commutators of operators},
Houston journal of mathematics 41(4):1163-1173

\bibitem{acg1} M.L. Arias, G. Corach, and M.C. Gonzalez,
\textit{Partial isometries in semi-Hilbertian spaces},
Linear Algebra Appl. \textbf{428} (7) (2008) 1460--1475.

\bibitem{acg2} M.L. Arias, G. Corach, and M.C. Gonzalez,
\textit{Metric properties of projections in semi-Hilbertian spaces},
Integral Equations Operator Theory \textbf{62}(1) (2008), 11--28.

\bibitem{acg3} {M.L. Arias, G. Corach, M.C. Gonzalez,} {Lifting properties in operator ranges,} Acta Sci. Math. (Szeged) 75:3-4(2009), 635-653.

\bibitem{bakfeki01} H. Baklouti, K. Feki and O. A. M. Sid Ahmed,
\textit{Joint numerical ranges of operators in semi-Hilbertian spaces},
Linear Algebra Appl. \textbf{555} (2018), 266--284.

\bibitem{bakfeki04} H. Baklouti, K.Feki and O.A.M. Sid Ahmed,
\textit{Joint normality of operators in semi-Hilbertian spaces},
Linear Multilinear Algebra \textbf{68} (2020), no. 4, 845--866.

\bibitem{B.F.P} P. Bhunia, K. Feki and K. Paul,
\textit{Numerical radius parallelism and orthogonality of semi-Hilbertian space operators and its applications},
Bull. Iran. Math. Soc. (2020), https://doi.org/10.1007/s41980-020-00392-8.

\bibitem{B.P.N} P. Bhunia, K. Paul, and R.K. Nayak,
\textit{On inequalities for $A$-numerical radius of operator},
Electronic Journal of Linear Algebra, \textbf{36} (2020), 143--157.

\bibitem{B.P.N.RIMA} P. Bhunia, R.K. Nayak and K. Paul,
\textit{Improvement of $A$-Numerical Radius Inequalities of Semi-Hilbertian Space Operators},
 Results Math 76, 120 (2021).
 
\bibitem{BNP2} P. Bhunia, R.K. Nayak and K. Paul, \textit{Refinement of  seminorm and numerical radius inequalities  of semi-Hilbertian space operators}, Math. Slovaca (2021), Accepted.

\bibitem{pkbsm} P. Bhunia and K. Paul,
\textit{New upper bounds for the numerical radius of Hilbert space operators},
Bul. Sci. Math., Volume 167, March 2021, 102959.

\bibitem{bc}{T. Bottazzi, C. Conde,} {Generalized numerical radius and related inequalities,} operators and matrices, 2021 (to appear)

\bibitem{branrov}{L. de Branges, J. Rovnyak,} {Square Summable Power Series,} Holt, Rinehert and Winston, New York, 1966.


\bibitem{Dou} R. G. Douglas,
\textit{On majorization, factorization and range inclusion of operators in Hilbert space},
Proc. Am. Math. Soc. \textbf{17} (1966) 413--416.

\bibitem{fekilaa} {K. Feki}, {On tuples of commuting operators in positive semidefinite inner product spaces}, Linear Algebra Appl. 603 (2020) 313-328.

\bibitem{feki01} K. Feki,
\textit{Spectral radius of semi-Hilbertian space operators and its applications},
Ann. Funct. Anal. 11, 929-946 (2020). \url{https://doi.org/10.1007/s43034-020-00064-y}

\bibitem{Feki.Sid} K. Feki and O.A.M. Sid Ahmed,
\textit{Davis-Wielandt shells of semi-Hilbertian space operators and its applications},
Banach J. Math. Anal. 14, 1281-1304 (2020). \url{https://doi.org/10.1007/s43037-020-00063-0}

\bibitem{feki03} {K. Feki}, {A note on the $A$-numerical radius of operators in semi-Hilbert spaces}, Arch. Math. 115, 535-544 (2020). \url{https://doi.org/10.1007/s00013-020-01482-z}

\bibitem{fpreprint03} {K. Feki}, {Some numerical radius inequalities for semi-Hilbert space operators}, J. Korean Math. Soc. 2021 (to appear)

\bibitem{faiot}{K.Feki,} {Generalized numerical radius inequalities of operators in Hilbert spaces,}  Adv. Oper. Theory (2020), \url{https://doi.org/ 10.1007/s43036-020-00099-x}.


\bibitem{fonghol} C.-K. Fong and J. A. R. Holbrook, Unitary invariant operator norms, Canad. J. Math., 35 (1983), 274-299.


\bibitem{A.K.2} O. Hirzallah and F. Kittaneh,
\textit{Numerical radius inequalities for several operators},
Math. Scand. 114 (2014), 110-119.

\bibitem{hou} {J. C. Hou, H. K. Du,} {Norm inequalities of positive operator matrices}, Integral Equations
Operator Theory 22 (1995), 281-294.


\bibitem{Ma.Se.Su} W. Majdak, N. A. Secelean and L. Suciu,
\textit{Ergodic properties of operators in some semi-Hilbertian spaces},
Linear Multilinear Algebra \textbf{61}(2) (2013), 139--159.



\bibitem{rout} {N. C. Rout, S. Sahoo, D. Mishra}, On $\mathbb{A}$-numerical radius inequalities for $2\times 2$ operator matrices, Linear  Multilinear Algebra  (2020)	\url{https://doi.org/10.1080/03081087.2020.1810201.}


\bibitem{saddi}{A. Saddi,} {$A$-Normal operators in Semi-Hilbertian spaces,} The Australian Journal of Mathematical Analysis and Applications, 9 (2012) 1-12.


\end{thebibliography}
\end{document}